\documentclass[12pt,reqno]{article}
\usepackage{amsmath,amsthm,amssymb}

\numberwithin{equation}{section}

\usepackage{tikz-cd}
\usepackage[all]{xy}
\usepackage{mathtext}
\usepackage[T1,T2A]{fontenc}
\usepackage[utf8]{inputenc}
\usepackage[russian,english]{babel}
\usepackage{amsfonts,longtable,amssymb,amsmath,latexsym, dsfont}
\usepackage{wrapfig}
\usepackage{xypic}
\usepackage{mathtools}
\usepackage{hyperref}

\title{	Projective Hessian and Sasakian manifolds}
\author{Pavel Osipov}

\usepackage{amsthm, amsmath, amssymb, amsfonts, amscd, amscd, geometry}
\newsavebox{\ssa}
\newcommand{\dt}{\frac{\partial}{\partial t}}

\newcommand{\g}{\mathfrak{g}}

\newcommand{\R}{\mathbb R}

\newcommand{\CC}{\mathbb{C}}

\theoremstyle{definition}
\newtheorem{theorem}{Theorem}[section]
\newtheorem{lemma}[theorem]{Lemma}
\newtheorem{proposition}[theorem]{Proposition}

\newtheorem{cor}[theorem]{Corollary}
\newtheorem{defin}[theorem]{Definition}
\newtheorem{example}[theorem]{Example}

\theoremstyle{remark}
\newtheorem{rem}{Remark}[section]
\usepackage{indentfirst}

\begin{document}
	
		\maketitle
		
		\begin{abstract}
			The Hessian geometry is the real analogue of the Kähler one. Sasakian geometry is an odd-dimensional counterpart to the Kähler geometry. In the paper, we study the connection between projective Hessian and Sasakian manifolds analogous to the one between Hessian and Kähler manifolds. In particular, we construct a Sasakian structure on $TM\times \R$ from a projective Hessian structure on $M$. We are interested in the case of invariant structure on Lie groups. We define semi-Sasakian Lie groups as a generalization of Sasakian Lie groups. Then we construct a semi-Sasakian structure on a group $G\ltimes \R^{n+1}$ for a projective Hessian Lie group $G$. We describe examples of homogeneous Hessian Lie groups and corresponding semi-Sasakian Lie groups. The big class of projective Hessian Lie groups can be constructed by homogeneous regular domains in $\R^n$. The groups $\text{SO}(2)$ and $\text{SU}(2)$ belong to another kind of examples. Using them, we construct semi-Sasakian structures on the group of the Euclidean motions of the real plane and the group of isometries of the complex plane.
			
		\end{abstract}

	\section{Introduction}
	A {\bfseries flat affine manifold} is a differentiable manifold equipped with a flat, torsion-free connection. Equivalently, it is a manifold equipped with an atlas such that all translation maps between charts are affine transformations (see \cite{FGH} or \cite{shima}). 
	A {\bfseries Hessian manifold} is an affine manifold with a Riemannian metric which is locally equivalent to a Hessian of a function. Any Kähler metric can be defined as a complex Hessian of a plurisubharmonic function. Thus, the Hessian geometry is a real analogue of the Kähler one. 
	
	A Kähler structure $(I,g^T)$ on $TM$ can be constructed by a Hessian structure $(\nabla,g)$ on $M$ (see \cite{shima}). The correspondence 
	$$
	\text{r}:\{\text{Hessian manifolds}\} \to \{\text{K\"ahler manifolds}\}
	$$
	$$
\ \ \ \	\ (M,\nabla,g)\  \ \to \ \ (TM,I,g^T)
	$$
	is called the {\bfseries r-map}. In particular, this map associates special Kähler manifolds to special real manifolds (see \cite {AC}). In this case, r-map describes a correspondence between the scalar geometries for supersymmetric theories in dimension {5 and 4.} See \cite{CMMS} for details on the r-map and supersymmetry.
	
	Hessian manifolds have many different application: in supersymmetry (\cite{CMMS}, \cite{CM}, \cite{AC}), in convex programming
	(\cite{N}, \cite{NN}), in the Monge-Ampère Equation (\cite{F1}, \cite{F2}, \cite{G}), in the WDVV equations (\cite{T}).

	A {\bfseries Riemannian cone} is a Riemannian manifold $(M\times \R^{>0}, t^2g_M+dt^2)$, where $t$ is a coordinate on $\R^{>0}$ and $g_M$ is a Riemannian metric on $M$. Riemannian cones have important applications in supegravity (\cite{ACDM}, 
	\cite{ACM}, \cite{CDM}, \cite{VDMV}). Geometry and holonomy of pseudo-Riemannian cones are studied in \cite{ACGL} and \cite{ACL}.

	The contact geometry is an odd-dimensional counterpart to the symplectic geometry. A manifold $M$ is {\bfseries contact} if and only if there exists a symplectic form $\omega$ on the cone $M\times \R^{>0}$ satisfying $\lambda_q \omega = q^2 \omega$, where $\lambda_q (m\times t) = m\times qt$. Moreover, if there exists an $\R^{>0}$-invariant complex structure $I$ on $M\times \R^{>0}$ such that $g=\omega(I\cdot, \cdot )$ is positive defined, that is, $(M\times\R^{>0},I,\omega)$ is a Kähler manifold then the manifold $M$ is called {\bfseries Sasakian}.	Any Sasakian manifold is a Riemannian cone (see \cite{OV}). Note that our definitions of contact and Sasakian manifolds are not standard but equivalent. See \cite{5sasaki} or \cite{BG} for standard definitions and \cite{ACHK} or \cite{OV} for equivalence of them.
		
				A {\bfseries Hessian cone} is a Hessian manifold $(M\times\R^{>0},\nabla,g=t^2g_M+dt^2)$ where $g_M$ is a metric on $M$, $t$ a coordianate on $\R^{>0}$, and $\nabla$ a flat torsion free connection satisfying
	$$
	\nabla \left(t\dt \right)=\text{Id}.
	$$
	We say that a Riemannian manifold $(M,g)$ is a {\bfseries projective Hessian manifold} if there exists a connection $\nabla$ on $M\times \R^{>0}$ such that $(M\times\R^{>0},\nabla,g=t^2g_M+dt^2)$ is a Hessian cone.
	
	We study a relation between projective Hessian and Sasakian manifolds analogous to the one between Hessian and Kähler manifolds. 
			\begin{theorem}
				Let $(M, g)$ be a projective Hessian manifold. Then $TM\times\R$ admits a structure of a Sasakian manifold.
			\end{theorem}
	This theorem is closely related to the r-map. Namely, we have a diagram 
	$$
	\begin{CD}
	M\times \R^{>0} @>\text{r}>> T(M\times\R^{>0})\\
	@AA con A @AA con A @.\\
	M @>>> TM\times\R
	\end{CD},
	$$
	where vertical arrows associate Riemannian cones to the corresponding Riemannian manifolds. The theorem implies that the Riemannian manifold $T(M\times\R^{>0})$ with the metric constructed by r-map is actually a cone over $TM\times \R$.

	Then we work with Lie algebras and groups equipped with invariant structures on them. There are different descriptions of an {\bfseries invariant affine structure} on a Lie algebra $\mathfrak{g}$: a torsion-free flat connection on $\mathfrak{g}$, an étale affine representation $\mathfrak{g}\to \mathfrak{aff}(\R^n)$, where $n$ is dimension of $\mathfrak{g}$, or a structure of a left symmetric algebra on $\mathfrak{g}$, that is, a multiplication on $\mathfrak{g}$ satisfying 
		$$
		XY-YX=[X,Y]
		$$
		and
		$$
		X(YZ)-(XY)Z=Z(XY)-(ZX)Y
		$$
		for any $X,Y,Z \in \mathfrak{g}$ (see \cite{burde} or \cite{Bu2}).
		
	Then we adapt r-map and Theorem 1.1. to the case of Lie algebras and groups. 
	
	\begin{theorem}[\cite{BD}]
		Let $\mathfrak{g}$ be an $n$-dimensional Lie algebra endowed with an affine structure $\nabla$ and $\eta$ the corresponding étale affine representation. Then there exists an invariant complex structure on $\mathfrak{g}\ltimes_\eta \R^n$. 
	\end{theorem} 

	It is immediately follows from the Theorem 1.2 that if a Lie group $G$ admits a left invariant affine structure then there exists a left invariant complex structure on a semidirect product $G\ltimes_\theta \R^n$. We show that $\theta$ is a linear part of the corresponding affine action of $G$ and get the following addition to results from \cite{BD}.   
		\begin{theorem}
			Let $G$ be a simply connected Lie group equipped with a left invariant affine structure, $\theta$ the linear part of the corresponding affine action of $G$. Then there exists a left invariant integrable complex structure $I$ on the group 
			$$
			G\ltimes_\theta \R^n\simeq TG
			$$ 
			such that $I$ swaps vertical and horizontal tangent subbundles.
		\end{theorem}

			Note that a Kähler structure on a the group $G\ltimes_{\theta^*} \left(\R^n\right)^*$ is constructed by an invariant Hessian structure on $G$ in \cite{medina}. The corresponding complex structure on $G\ltimes_{\theta^*} \left(\R^n\right)^*$ depends on the Hessian metric on $G$. In our case, the complex structure on $G\ltimes_\theta \R^{n}$ depends only on the affine structure on $G$. 
			
	{\bfseries A Hessian Lie group} $(G,\nabla, g)$ is a Lie group $G$ endowed with a left invariant affine structure $\nabla$ and a left invariant Hessian metric $g$.
			
		\begin{theorem}
			Let $G$ be an $n$-dimensional simply connected Lie group equipped with a left invariant affine structure $\nabla$ and $\theta$ the linear part of the corresponding affine action of $G$. Then there exists a left invariant Kähler metric on $G\ltimes_\theta \R^n=TG.$
		\end{theorem}
		The construction from Theorem 1.4. is a homogeneous analogue of r-map.
	
	For adapting Theorem 1.1 to the case of Lie algebras we introduce definitions of semi-sasakian Lie algebras and groups. 
	
	Let $\g$ be a $2n+1$-dimensional Lie algebra, $D$ a  derivation of $\g$, and $\omega$ 2-form on $\g$ and $\eta$ a 1-form on $\g$. If 
	$$
	d\omega =0, \ \ \ \omega+D^*\omega-d\eta = 0  \ \ \ \text{and} \ \ \ \eta\wedge \omega^n \ne 0
	$$
	then we say that a collection $(\g,D, \omega,\eta)$ is a {\bfseries semi-contact Lie algebra}. Note that if $D=0$ then we have the pair $(\g,\eta)$ such that $\eta(d\eta)^n\ne 0$ i.e. $(\g,\eta)$ is a contact Lie algebra.
	 
	The definition of semi-contact Lie algebras is closely related to the definition of lcs Lie algebras. A {\bfseries locally conformally symplectic} (shortly, {\bfseries lcs}) {\bfseries algebra}  is a $2n$-dimensional Lie algebra $\g$ endowed with a 2-form $\omega$ and 1-form $\vartheta$ such that 
	$$
	\Omega^n\ne 0, \ \ \ d\vartheta=0, \ \  \text{and} \ \ \ d\Omega = \vartheta \omega \ \ \ \text{(See \cite{BD} or \cite{ABP})}.
	$$
	According to results of \cite{ABP}, a collection $(\g,D, \omega,\eta)$ is a semi-contact Lie algebra if and if $(\R\ltimes_D \g, \omega+\vartheta \wedge \eta)$ is an lcs Lie algebra, where $\vartheta(r\ltimes_D v)=r$, and any lcs Lie algebra arises by this way. Thus, there is a one-to-one correspondence between lcs and semicontact Lie algebras. 	If $(\g,\Omega,\vartheta)$ is an lcs Lie algebra then we say that $(G,\Omega,\vartheta)$ is an {\bfseries lcs Lie group}.

	A {\bfseries semi-contact Lie group} is a triple $(G,\theta, \Omega)$, where $G$ is a Lie algebra $\theta$ be an action of $\R^{>0}$ on $G$ and $\omega$ a symplectic form on $\R^{>0}\ltimes_\theta G$ such that $\omega$ is invariant with respect to the left action of $G$ and satisfying $\lambda_q^* \omega =q^2 \omega$, where $\lambda_q x= (q\ltimes 1 )(x)$. If $\theta=id$ then $G,\Omega$ is a contact Lie group. See \cite{D} for more information on contact Lie algebras and groups.

		\begin{theorem}
			Let $\g$ be a Lie algebra of left invariant vector fields on a Lie group $G$, $D$ a derivation of $\g$, $\theta=\exp D$ the corresponding automorphism of $G$, $\omega$ and $\eta$ left invariant $2$-form and $1$-form correspondingly. Then the following conditions are equivalent:
			\begin{itemize}
				\item [(i)] $(\g,D,\omega,\eta)$ is a semicontact Lie algebra.
				\item[(ii)] $\left(\R^{>0}\ltimes_\theta G, \Omega=\omega+t^{-1}\eta\wedge dt,2t^{-1}dt
				\right)$ is an lcs Lie group.
				\item[(iii)] $(G,\theta, \hat\Omega=t^2\omega+tdt\wedge\eta)$ is a semi-contact Lie group.

			\end{itemize}
		\end{theorem}
		
		A {\bfseries locally conformally K\"ahler} (shortly, {\bfseries lck}) Lie algebra is an lcs $(\mathfrak{h},\Omega,\vartheta)$ Lie algebra endowed with a complex structure $I$ such that 
		$$
		g:=\Omega(*,I*)
		$$ 
		is a positive definite symmetric bilinear form (see \cite{HK}). The Lie group corresponding to an lck Lie algeba is called an lck Lie group. Note that homogeneous lck manifolds of reductive Lie groups are classified in \cite{ACHK}. 
		
	A {\bfseries semi-Sasakian Lie algebra} is a semi-contact Lie algebra $(\g,\eta,\omega)$ with a left-invariant integrable  almost complex structure $I$ on $(\g\ltimes_D \R)$ such that 
	$$
	g=\hat\Omega(*,I*)
	$$ 
	is a positive definite symmetric bilinear form, where $\hat\Omega=\omega+\vartheta \wedge \eta$. As above, there is a one-to one correspondence between semi-Sasakian algebras $(\g,D, \omega,\eta)$ and lck algebras $(\R\ltimes_D \g, \omega+\vartheta \wedge \eta)$.

	A {\bfseries semi-Sasakian Lie group} is a semi-contact Lie group $(G,\theta,\hat\Omega)$ with a left invariant complex structure $I$ on $\R^{>0}\ltimes_\theta G$ such that $(\R^{>0}\ltimes_\theta G,\hat\Omega,I)$ is a K\"ahler manifold. If $\theta=id$ then $G$ is called a {\bfseries Sasakian Lie group}. Note that in this case $g=t^2g_G+dt^2$ according to Proposition 3.5.

	\begin{cor}
		Let $\g$ be a Lie algebra of the invariant vector fields on a Lie group $G$, $D$  the derivation of $\g$, $\theta=\exp D$ the corresponding automorphism of $G$, $I$ a complex structure on $\R\ltimes \g$, $\omega$ and $\eta$ left invariant $2$-form and $1$-form correspondingly. Then the following conditions are equivalent:
		\begin{itemize}
			\item [(i)] $(\g,D,\omega,\eta,I)$ is a semi-Sasakian Lie algebra.
			\item[(ii)] $\left(\R^{>0}\ltimes_\theta G, \Omega=\omega+t^{-1}\alpha\wedge dt,2t^{-1}dt,I \right)$ is an lck Lie group.
			\item[(iii)] $(G,\theta, \hat\Omega=t^2\omega+tdt\wedge\eta,I)$ is a semi-Sasakian Lie group.

		\end{itemize}
	\end{cor}

	A Lie algebra $\mathfrak{g}$ is called {\bfseries projective} if there is an invariant affine structure $\nabla$ on $\mathfrak{g}\times \R$  such that 
	$$
	\nabla_X E =\nabla_E X = X,
	$$
	where $X\in \mathfrak{g}$ and $E\in \R$. A Lie group is called {\bfseries projective} if the corresponding Lie algebra is projective.  Note that if $G$ is a projective Lie group then there is an invariant affine structure on $G\times \R^{>0}$. A {\bfseries projective Hessian Lie group} $(G,g_G)$ is a projective Lie group $G$ endowed with a left-invariant Riemannian metric $g_G$ such that $(G\times\R^{>0},\nabla,g=t^2g_G +dt^2)$ is a Hessian cone, where $t$ is a coordinate on $\R^{>0}$ and $\nabla$ the affine connection on $G\times \R^{>0}$ corresponding to the projective structure on $G$.

		\begin{theorem}
		Let $G$ be an $n$-dimensional simply connected projective Hessian Lie group and $\theta$ the linear part of the corresponding affine representation of $G \times \R^{>0}$. Then there exists a structure of a semi-Sasakian Lie group on  $G\ltimes_\theta \R^{n+1}$. Moreover, $G\ltimes_\theta \R^{n+1}\simeq 
		TG\times \R$. 
	    \end{theorem}
	    
		Note that Theorem 1.7 is a homogeneous analogue of Theorem 1.1.

 In section 6, we explain why Theorem 1.4 and Theorem 1.5 generalize the known construction of homogeneous K\"ahler domains by a homogeneous convex regular domains. According to \cite{vinb}, an invariant Hessian metric can be constructed on any homogeneous convex regular domain. Moreover, a homogeneous K\"ahler domain can be associated with a homogeneous convex regular cone (see \cite{shima}). We can get the same construction applying Theorem 1.4 to a Lie group acting simply transitively on a homogeneous regular convex domain. Such group always exists and it is called a {\bfseries Lie group associated with} a homogeneous regular domain. Any Lie group associated with a homogeneous regular domain is both a Hessian Lie group and a projective Lie group. 

The groups $\text{U}(1)=\text{SO}(2)$ and $\text{SU}(2)$ belong to another kind of projective Hessian Lie groups. Using them, we can construct a semi-Sasakian structure on the Euclidean group $\text{E}(2)$ and the group of isometries of the complex plane $\mathbb{C}^2$. Any Sasakian group is obviously semi-Sasakian. The Sasakian groups of dimension $n \le 5$ are classified in \cite{5sasaki}. Using this classification, we verify that the semi-Sasakian group $\text{E}(2)$ does not admit a Sasakian structure.  

Any Lie group associated with a homogeneous regular domain admits both structures: Hessian and projective Hessian. The group $\text{SO}(2)$ admits both structures too. However, not any projective Hessian group is Hessian. The group $\text{SU}(2)$ is not Hessian just because the sphere $S^3$ does not admit any flat affine structure. However, the group $\text{SU}(2)$ admits an invariant projective structure, since there is an invariant affine structure on $\text{SU}(2)\times\R^{>0}$. Thus, the natural question arises: does the existence of $G$-invariant selfsimilar structure on $G\times\R^{>0}$ implies the existence of $G\times\R^{>0}$-invariant Hessian structure on $G\times \R^{>0}$? The answer is positive when $G$ is Lie group associated with a homogeneous regular domain or $\text{U}(1)$. We show that $\text{SU}(2)\times \R^{>0}$ does not admit an invariant Hessian structure.

	\section{Hessian and Kähler structures}
	\begin{defin}
		 A {\bfseries flat affine manifold} is a differentiable manifold equipped with a flat, torsion-free connection. Equivalently, it is a manifold equipped with an atlas such that all translation maps between charts are affine transformations (see \cite{FGH}).  
	\end{defin}

	\begin{defin}
		A Riemannianian metric $g$ on a flat affine manifold $(M,\nabla)$ is called to be a {\bfseries Hessian metric} if $g$ is locally expressed by a Hessian of a function
		$$
		g=\text{Hess} \varphi=\nabla d \varphi =\frac{\partial^2}{\partial x^i \partial x^j} dx^i dx^j,
		$$
		where $x^1,\ldots, x^n$  are flat local coordinates. A {\bfseries Hessian manifold} $(M,\nabla,g))$ is a flat affine manifold $(M,\nabla)$ endowed with a Hessian metric $g$. (see \cite{shima}). 
	\end{defin}
	
	Let $U$ be an open chart on a flat affine manifold $M$, functions $x^1,\ldots, x^n$ be affine coordinates on $U$, and $x^1,\ldots, x^n, y^1, \ldots, y^n$ be the corresponding coordinates on $TU$. Define the complex structure $I$ by $I(\frac{\partial}{\partial x^i})=\frac {\partial} {\partial y^i}$. Corresponding complex coordinates are given by $z^i=x^i+\sqrt {-1}y^i$. The complex structure $I$ does not depend on a choice of flat coordinates on $U$. Thus, in this way, we get a complex structure on the $TM$.
	
	Let $\pi : TM \to M$ be a natural projection. Consider a Riemannian metric $g$ given locally by
	$$
	f_{i,j} dx^idx^j.
	$$
	Define the Hermitian metric $g^T$ on $TM$ by
	$$
	\pi^* f_{i,j} dz^id\bar z^j.
	$$
	\begin{proposition}[\cite{shima}, \cite{AC}]
		Let $M$ be a flat affine manifold, $g$ and $g^T$ as above. Then the following conditions are equivalent:
			\begin{itemize}
				\item[(i)]$g$ is a Hessian metric.
		
		\item[(ii)] $g^T$ is a Kähler metric. 
	\end{itemize}
		Moreover, if
		$
		g=\text{Hess} \varphi
		$ 
		locally then $g^T$ is equal to a complex Hessian
		$$
		g=\text{Hess}_\mathbb{C} (4\pi^*\varphi).
		$$
	\end{proposition}

	\begin{defin}
		The metric $g^T$ is called {\bfseries Kähler metric associated with  $g$}. The correspondence that associates the Kähler manifold $(TM,g^T)$ with a Hessian manifold $(M, g)$ is called {\bfseries r-map} (see \cite{AC}).
	\end{defin}
	
	\begin{proposition}
		Let $(M,g)$ be a Hessian manifold, $I$ the corresponding complex structure on $TM$, and $\pi : TM \to M$ a projection. Then the associated Kähler metric $g^T$ equals
		\begin{equation}
		h(X,Y)=\pi^*g(X,Y)+\pi^*g(IX,IY)+\sqrt{-1}\pi^*g(IX,Y)-\sqrt{-1}\pi^*g(X,IY).
		\end{equation}
	\end{proposition}
	\begin{proof}
		If 
		\begin{equation}
		g=f_{i,j} dx^i dx^j
		\end{equation}
		then

		$$
		g^T=\pi^* f_{i,j} dz^i d\bar z^j=\pi^* f_{i,j} d(x^i+\sqrt{-1}y^i) d(x^j-\sqrt{-1}y^j)=
		$$
		\begin{equation}
		=\pi^* f_{i,j} dx^i dx^j+\pi^* f_{i,j} dy^idy^j+\sqrt{-1}\pi^* f_{i,j} dy^idx^j-\sqrt{-1}\pi^* f_{i,j} dx^i dy^j.
		\end{equation}
		It is enough to check the identity 
		$$
		h(X,Y)=g^T(X,Y)
		$$
		on pairs of basis vectors. For any $i\in \{1,\ldots,n\}$ we have
		$$
		\pi^*g\left(\frac{\partial}{\partial y^i},\ldots \right)=0,
		$$
		moreover,
		$$
		I \frac{\partial}{\partial x^i}=\frac{\partial}{\partial y^i}.
		$$
		Hence, by (2.1),
		$$
		h\left(\frac{\partial}{\partial x^i},\frac{\partial}{\partial x^j}\right)=\pi^* g\left(\frac{\partial}{\partial x^i},\frac{\partial}{\partial x^j}\right)
		$$
		and, by (2.2),
		$$
		\pi^* g\left(\frac{\partial}{\partial x^i},\frac{\partial}{\partial x^j}\right)=f_{i,j}.
		$$
		On the other hand, by (2.3),
		$$
		g^T\left(\frac{\partial}{\partial x^i},\frac{\partial}{\partial x^j}\right)=f_{i,j}.
		$$
		Thus, we get
		$$
		g^T\left(\frac{\partial}{\partial x^i},\frac{\partial}{\partial x^j}\right)=h\left(\frac{\partial}{\partial x^i},\frac{\partial}{\partial x^j}\right).
		$$ 
		Checking for the pairs $\left(\frac{\partial}{\partial x^i},\frac{\partial}{\partial y^j}\right)$ and $\left(\frac{\partial}{\partial y^i},\frac{\partial}{\partial y^j}\right)$ is similar.
	\end{proof}

	\section{Projective Hessian and Sasakian manifolds}
	
	\begin{defin}
	A {\bfseries radiant manifold} $(C,\nabla, \xi)$ is a flat affine manifold $(C,\nabla)$ endowed with a vector field $\xi$ satisfying
	\begin{equation}
	\nabla \xi =\text{Id}.
	\end{equation}
	
	Equivalently, it is a manifold equipped with an atlas such that all translation maps between charts are linear transformations (see e.g. \cite{Go}).
	\end{defin}

	\begin{proposition}[\cite{Go}]
		Let $t$ be a coordinate on $\R^{>0}$ and $(M\times \R^{>0},\nabla,t\dt)$ a radiant manifold. Consider a natural action of $\R^{>0}$ on $M\times \R^{>0}$. Then the connection $\nabla$ is $\R^{>0}\text{-invariant}$.
	\end{proposition}

			\begin{defin}
			A {\bfseries Hessian cone} is a Hessian manifold $(M\times\R^{>0},\nabla,g=t^2g_M+dt^2)$ such that $(M\times \R^{>0},\nabla,t\dt)$ is a radiant manifold, where $t$ is a coordinate on $\R^{>0}$.
			We say that a Riemannian manifold $(M,g_M)$ is a {\bfseries projective Hessian manifold} if there exists a connection $\nabla$ on $M\times \R^{>0}$ such that $(M\times\R^{>0},\nabla,g=t^2g_M+dt^2)$ is a Hessian cone.
		\end{defin}

	\begin{defin}
		A {\bfseries Sasakian manifold} is a Riemannian manifold $(M,g_M)$ such that the cone metric $t^2 g_M+ dt^2$ on $M\times \R^{>0}$ is Kähler with respect to a dilation invariant complex structure.
	\end{defin}
	
	\begin{proposition}[\cite{OV}]
		Let $(M\times \R^{>0}, g, I)$ be a Kähler manifold. For any $q\in \R^{>0}$ consider a map 
		$
		\lambda_q : M\times \R^{>0} \to M\times \R^{>0}
		$
		defined by
		$
		\lambda_q (m\times t)=m\times qt.
		$ 
		If $\lambda_q^* g=q^2 g$ then $g=t^2g_M+dt^2$ and $(M,g_M)$ is a Sasakian manifold. 
	\end{proposition}

	There exists a decomposition	
	$$
	T(M\times\R^{>0})=TM \times T\R^{>0}=TM\times \R^{>0} \times \R.
	$$
	If $M\times\R^{>0}$ possess a Hessian structure then, by Proposition 2.3, $T(M\times\R^{>0})$ admits a Kähler structure.
	
	\begin{proposition}
		Let $(M\times \R^{>0}, \nabla, g)$ be a Hessian cone and $g^T$ the associated Kähler metric on $T(M\times \R^{>0})$ . Consider $T(M\times\R^{>0})=TM\times \R \times \R^{>0}$ as a cone over $TM\times \R$. Then for any $q\in\R^{>0}$ we have $\mu_q^* g=q^2 g$, where the map 
		$
		\mu_q : TM\times\R\times \R^{>0} \to TM\times \R\times \R^{>0}
		$
		is defined by
		$
		\mu_q (m\times s\times t)=m\times s\times  qt.
		$  
	\end{proposition}
	
	\begin{proof}
		We have the commutative diagram 
		$$
		\begin{CD}
		T(M\times\R^{>0}) @>\mu_q>> T(M\times\R^{>0})\\
		@VV\pi V @VV\pi V  @.\\
		M\times\R^{>0} @>\lambda_q>> M\times\R^{>0}
		\end{CD},
		$$
		where $\mu_q$ and $\lambda_q$ are multiplications of the coordinate on $\R^{>0}$ by $q$. By Proposition 2.5, we have 
		\begin{equation}
		g^T(X,Y)=\pi^*g(X,Y)+\pi^*g(IX,IY)+\sqrt{-1}\pi^*g(IX,Y)-\sqrt{-1}\pi^*g(X,IY).
		\end{equation}	
		Since the diagram is commutative, it follows that \begin{equation}
		\mu_q^*\pi^*=\pi^*\lambda_q^*.
		\end{equation} 
		Moreover, $g$ is a cone metric. Hence,
		\begin{equation}
		\lambda_q^* g= q^2 g.
		\end{equation}
		It follows from (3.2), (3.3), and (3.4) that 
		$$
		\mu_q^* g^T(X,Y)=q^2 g^T(X,Y).
		$$  
	\end{proof}

	\begin{theorem}
		Let $(M, g_M)$ be a projective Hessian manifold. Then $TM\times\R$ admits a structure of a Sasakian manifold.
	\end{theorem}
	\begin{proof}

	Define a K\"ahler structure $(g^T,I)$ on $T(M\times\R^{>0})=TM\times\R\times\R^{>0}$ as above. 	By proposition 3.2, the connection $\nabla$ is $\R^{>0}$ invariant. Therefore, the constructed by $\nabla$ complex structure $I$ is $\R^{>0}$-invariant.  Then, by Proposition 3.4 and Proposition 3.5, $TM\times \R$ admits a structure of a Sasakian manifold. 
	\end{proof}

	\section{Affine representations and flat torsion free connections on Lie groups}
		
		The group of affine transformations $\text{Aff} (\mathbb{R}^n)$ is given by the matriсes of the form
		$$
		\begin{pmatrix}
		A & a\\
		0 & 1
		\end{pmatrix}
		\in \text{GL}(\R^{n+1}),
		$$
		where $A\in \text{GL}(\R^n)$, and $a\in \R^n$ is a column vector. The corresponding Lie algebra $\mathfrak{aff}(\R^n)$ is given by matrices of the form
		$$
		\begin{pmatrix}
		A & a\\
		0 & 0
		\end{pmatrix}
		\in \mathfrak{gl}(\R^{n+1}).
		$$
		The commutator of $\mathfrak{aff}(\R^n)$ is equal to
		$$
		\left[
		\begin{pmatrix}
		A & a\\
		0 & 0
		\end{pmatrix},
		\begin{pmatrix}
		B & b\\
		0 & 0
		\end{pmatrix}
		\right]
		=
		\begin{pmatrix}
		[A,B] & A(b)-B(a)\\
		0 & 0
		\end{pmatrix}.
		$$
		Algebra $\mathfrak{aff} (\R^n)$ is the semidirect product $\mathfrak{gl}(\R^n)\ltimes \R^n$, where the commutator is given by
		$$
		[A\ltimes a,B \ltimes b]=[A,B]\ltimes (Ab-Ba).
		$$
		Group $\text{Aff}(\R^n)$ is the semidirect product $\text{GL}(\R^n)\ltimes \R^n$, where multiplication is given by 
		$$
		(A\ltimes a) (B \ltimes B)=AB\ltimes (a+Ab).
		$$

		\begin{defin}
			An affine representation is called {\slshape étale} if there exists a point $x \in \R^n$ such that the orbit of $x$ is open and the stabilizer of $x$ is discrete. 
		\end{defin}
		
		\begin{theorem}[\cite{burde} or \cite{Bu2}]
			Let $G$ be a Lie group. There is a correspondence between left invariant torsion-free flat connections and étale affine representations. Moreover, if the connection is complete then the corresponding étale affine representation acts simply transitive on $\R^n$.

		\end{theorem}
		
		\begin{proof}
			
			Choosing a basis, identify $\mathfrak{g}$ with $\R^n$. Then consider $\nabla_X$ as a linear endomorphism of $\R^n$ for any $X\in \mathfrak{g}$. The corresponding to $\nabla$ étale  representation is given by 
			$$
			\eta : \mathfrak{g} \to \mathfrak{aff} (\R^n),
			$$
			
			\begin{equation}		
			\eta(X)  = 
			\begin{pmatrix}
			\nabla_X & X\\
			0 & 0
			\end{pmatrix}
			\in \mathfrak{aff}(\R^n)\subset \mathfrak{gl}(\R^{n+1})		
			\end{equation}
			(see \cite{burde} or \cite{Bu2} for details).
		\end{proof} 
		
		For any $X \in \mathfrak{g}$ we can consider $\nabla_X$ on the Lie algebras $\mathfrak{g}$ as a linear automorphism of the vector space $\mathfrak{g}$. Thus, the linear automorphism 
		$$
		\exp \nabla_X=\text{id} +\frac{\nabla_X}{1!}+\frac{\nabla_X\nabla_X}{2!}+\frac{\nabla_X\nabla_X\nabla_X}{3!}+\ldots 
		$$
		is well defined.
		\begin{proposition}
			Let $\nabla$ be a left invariant flat torsion-free connection on a simply connected Lie group $G$ and $\tau$ the corresponding étale  affine representation of $G$. If $X \in \mathfrak{g}$ then the linear part of $\tau(\exp X)$ is equal to $\exp \nabla_X$.
		\end{proposition}
		\begin{proof}
			By (4.1), we have
			$$
			\tau(\exp X)=\exp 
			\begin{pmatrix}
			\nabla_X & X\\
			0 & 0
			\end{pmatrix} =
			\begin{pmatrix}
			\exp \nabla_X & (\exp \nabla_X )(X)\\
			0 & 1
			\end{pmatrix}.
			$$
			Hence, the linear part of $\tau(\exp X)$ is equal to $\exp \nabla_X$.
		\end{proof}

	Fix notations: let $\mathfrak{g}$ be a Lie algebra; $G$ the corresponding simply connected Lie group; $\eta$ an étale  affine representation $\mathfrak{g} \to \mathfrak{aff}$; $\theta$ the linear part of the corresponding affine representation of $G$; $i$ an identification $\mathfrak{g} \to \R^n$.
	 
	Define an almost complex structure $I$ on $\mathfrak{g} \ltimes_\eta \R^n$ by the rule
		$$
		I(X\ltimes_\eta Y)= -i^{-1} Y \ltimes_\eta iX. 
		$$
		
			\begin{theorem}[\cite{CO} or \cite{BD}]
				Let $\mathfrak{g}$, $\eta$, $I$ be as above. Then the almost complex structure $I$ is integrable.
			\end{theorem}
			
			\begin{defin}
				The algebra $\mathfrak{g}\ltimes_\eta \R^n$ is called {\bfseries associated with the connection} $\nabla$ and denoted by $\mathfrak{g_\nabla}$.
			\end{defin}
		
		\begin{proposition}
			
			Let $\mathfrak{g}, \nabla, \theta, \eta$ be as above and $G_\nabla$ the simply connected Lie group corresponded to $\mathfrak{g}_\nabla$. Then 
			$$G_\nabla=G\ltimes_\theta \R^n.$$
			Moreover, there is an identification $TG=G \ltimes_\theta \R^n$ such that fields of the form $0\oplus I\mathfrak{g}\subset T(TG)$ are vertical, that is, they lie in $\ker d\pi$, where $\pi: TG \to G$ is a projection.
			
		\end{proposition}

		\begin{proof}
			By the definition, $\mathfrak{g}_\nabla$ is isomorphic to $\mathfrak{g} \ltimes_\eta \R^n$. The corresponding Lie group is equal to the semidirect product $G \ltimes \R^n$ with respect to the action such that the element $\exp X$ acts on $\R^n$ by $\exp \nabla_X$. By Proposition 4.6, this action equals $\theta$. Thus, the corresponding Lie group $G_\nabla$ is equal to $G \ltimes_\theta \R^n$

			Using the trivialization of the tangent bundle of $G$ by the flat connection $\nabla$, we identify $TG$ with $G \times \R^n$. Define the multiplication by
			$$
			(g_1\times X_1)(g_2\times X_2)=(g_1 g_2) \times (X_1 + \theta(g_1) (X_2)).
			$$
			This multiplication is equal to the multiplication on $G_\nabla =G\ltimes_\theta \R^n$. Thus, the group $TG$ with this multiplication is isomorphic to $G_\nabla$. Moreover, the left invariant fields corresponding to the subalgebra $0\oplus I \mathfrak{g}$  are actually vertical.
			
		\end{proof}

	\begin{theorem}
		Let $G$ be a simply connected Lie group equipped with a left invariant affine structure, $\mathfrak{g}$ the corresponding Lie algebra, and $\theta$ the linear part of the corresponding affine action of $G$. Then there exists a left invariant integrable complex structure $I$ on the group 
		$$
		G\ltimes_\theta \R^n\simeq TG
		$$ 
		defined in Proposition 4.6 such that $I$ swaps vertical and horizontal tangent subbundles.
	\end{theorem}
	\begin{proof}
		The theorem follows from Theorem 4.4 and Proposition 4.6.
	\end{proof}

	\section{Semi-sasakian Lie algebras and groups}
		\begin{defin}
			A {\bfseries locally conformally symplectic} (shortly, {\bfseries lcs}) {\bfseries algebra}  is a $2n$-dimensional Lie algebra $\g$ endowed with a 2-form $\omega \in \Lambda^2 \g^*$ and 1-form $\vartheta\in\g^*$ such that 
			$$
			\Omega^n\ne 0, \ \ \ d\vartheta=0, \ \  \text{and} \ \ \ d\Omega = \vartheta \wedge\omega.
			$$
			If $(\g,\Omega,\vartheta)$ is a {\bfseries lcs Lie algebra} then we say that $(G,\Omega,\vartheta)$ is lcs Lie group. Here we consider $\Omega$ and $\vartheta$ as tensors on left invariant vector fields.  
		\end{defin} 
		
		\begin{proposition}[\cite{ABP}]
			Any lcs Lie algebra $(\mathfrak{h},\Omega, \vartheta)$ takes a form $(\R\ltimes_D \g, \Omega=\omega+\vartheta \wedge \eta)$, where $\g$ is a Lie algebra, $D$ a derivation of $\g$, $\vartheta$ a 1-form on $\R\ltimes_D \g$ defined by $\vartheta(r\ltimes_D v)=r$, $\omega$ and $\eta$ are $G$-invariant 2-form and 1-form correspondingly satisfying
			\begin{equation}
			d^\g\omega =0, \ \ \ \omega+D^*\omega-d^\g\eta = 0  \ \ \ \text{and} \ \ \ \eta\wedge \omega^n \ne 0.
			\end{equation}
			On the other side, if a collection $(\R\ltimes_D \g, \Omega=\omega+\vartheta \wedge \eta)$ satisfies (5.1) then it defines an lcs algebra.
		\end{proposition}
		\begin{defin}
			Let $\g$ be a $2n+1$-dimensional Lie algebra, $D$ a derivation of $\g$, 
			$\omega$ a 2-form on $\g$, and $\eta$ a 1-form on $\g$. If 
			$$
			d\omega =0, \ \ \ \omega+D^*\omega-d\eta = 0  \ \ \ \text{and} \ \ \ \eta\wedge \omega^n \ne 0
			$$
			then we say that a collection $(\g,D, \omega,\eta)$ is a {\bfseries semi-contact Lie algebra}. 
		\end{defin}
		
		According to Definition 5.3 and Proposition 5.2, $(\g,D, \omega,\eta)$ is a semi-contact Lie algebra if and if $(\R\ltimes_D \g, \omega+\vartheta \wedge \eta)$ is lcs, where $\vartheta$ is as above. Thus, there is a one-to-one correspondence between lcs and Lie semicontact algebras.

		\begin{defin}
			A {\bfseries semi-contact Lie group} is a triple $(G,\theta, \hat\Omega)$, where $G$ is a Lie algebra, $\theta$ an action of $\R^{>0}$ on $G$, and $\hat\Omega$ a symplectic form on $\R^{>0}\ltimes_\theta G$ such that $\omega$ is invariant with respect to the left action of $G$ and satisfying $\lambda_q^* \hat\Omega =q^2 \hat\Omega$, where $\lambda_q x= (q\ltimes 1 )(x)$.
		\end{defin}
		
		\begin{proposition}
			Let $G$ be a Lie group, $t$ a coordinate on $\R^{>0}$, $\theta$ an automorphism of $G$, and $\Omega$ a 2-form on $\R^{>0}\ltimes_\theta G$. Then the collection $(G,\theta,\hat\Omega)$ is a semi-contact Lie group if and only if $(\R^{>0}\ltimes_\theta G,{\Omega}=t^{-2}\hat\Omega,2tdt)$ is an lcs Lie group.
		\end{proposition}
		
		\begin{proof}
			We have
			$$
		 d\Omega=d\left(t^{-2}\hat\Omega\right)= 2t^{-3} dt\wedge\hat\Omega +t^{-2} d\hat\Omega=\left(2t^{-1}dt\right)\wedge {\Omega}+t^{-2}d\hat\Omega.
			$$
			Thus, $d{\Omega}=2tdt\wedge {\Omega}$ if and only if $d\hat\Omega=0$. 
			
			Moreover $\hat\Omega=t^2  \Omega$ satisfies $\lambda_q^*\hat\Omega = q^2\hat\Omega$ if and only if $\lambda_q\Omega={\Omega}$.
		\end{proof}
		\begin{cor}
			Let $(G,\theta, \Omega)$ be a semi-contact Lie group. Then the form $\Omega$ on $\R^{>0}\ltimes_\theta G$ can be written as $\Omega = t^2\omega + t dt\wedge\eta$, where $\omega$ and $\eta$ are $G$-invariant 2-form and 1-form.
		\end{cor}
		
		\begin{proof}
			By Proposition 5.5, $\left(\R^{>0}\ltimes_\theta G, \hat{\Omega}=t^{-2}\Omega\right)$ is lcs. Let $D$ be a derivation of $\g$ such that $\exp D= \theta$. Consider an identification of $\R\ltimes_D \g$ with the left invariant fields on $\R^{>0}\ltimes_\theta G$ such that
			$$
			1\ltimes_D 0= t \dt.
			$$
			Let $\vartheta$ be as in Proposition 5.2. Then
			$$
			\vartheta = t^{-1}dt.
			$$
			By Proposition 5.2, we have  $$\hat \Omega =\omega +t^{-1}dt \wedge\eta.$$ Thus, we have 
			$$
			\Omega= t^2\omega +tdt\wedge\eta.
			$$
		\end{proof}
		
		\begin{theorem}
			Let $\g$ be a Lie algebra of left invariant vector fields on a Lie group $G$, $D$ a derivation of $\g$, $\theta=\exp D$ the corresponding automorphism of $G$, $\omega$ and $\eta$ left invariant $2$-form and $1$-form correspondingly. Then the following conditions are equivalent:
			\begin{itemize}
				\item [(i)] $(\g,D,\omega,\eta)$ is a semicontact Lie algebra.
				\item[(ii)] $\left(\R^{>0}\ltimes_\theta G, \Omega=\omega+t^{-1}\eta\wedge dt,2t^{-1}dt
				 \right)$ is an lcs Lie group.
				\item[(iii)] $(G,\theta, \hat\Omega=t^2\omega+tdt\wedge\eta)$ is a semi-contact Lie group.

			\end{itemize}
		\end{theorem}
		
		\begin{proof}
			$\text{(ii)}\Leftrightarrow\text{(iii)}$\ It follows from Proposition 5.5.
			
			$\text{(i)}\Leftrightarrow\text{(ii)}$ Consider an identification of $\R\ltimes_D \g$ with the left invariant fields on $\R^{>0}\ltimes_\theta G$ such that
			$$
			1\ltimes_D 0 = \frac{1}{2} t \dt.
			$$
			Let $\vartheta$ be as in Proposition 5.2. Then
			$$
			\vartheta = 2t^{-1} dt.
			$$
			By Proposition 5.2, $(\g,D,\omega,\eta)$ is a semicontact Lie algebra if and only if   $\left(\R\ltimes_D \g,\Omega,\vartheta=2t^{-1}dt\right)$ is an lcs Lie algebra. By the definition,  $(G,\theta, \omega,\eta)$ is a semi-contact Lie group if and only if  $\left(\R\ltimes_D \g,\Omega,\vartheta\right)$ is lcs Lie algebra. 
		\end{proof}
		\begin{defin}
			A {\bfseries locally conformally K\"ahler} (shortly, {\bfseries lck}) Lie algebra is an lcs $(\mathfrak{h},D,\Omega,\vartheta)$ Lie algebra endowed with a complex structure $I$ such that 
			$$
			g:=\Omega(*,I*)
			$$ 
			is a positive definite symmetric bilinear form (see \cite{HK}). The Lie group corresponding to an lck Lie algebra is called an lck Lie group.
		\end{defin}		
		\begin{defin}
			A semi-Sasakian Lie algebra is a semi-contact Lie algebra $(\g,D,\omega,\eta)$ with a left-invariant integrable  almost complex structure $I$ on $(\g\ltimes_D \R)$ such that 
			$$
			g:=\hat\Omega(*,I*)
			$$ 
			is a positive definite symmetric bilinear form, where $\hat\Omega=\omega+\vartheta\wedge \eta$ and $\vartheta$ is as in Proposition 5.2.
		\end{defin}
		
		\begin{defin}
			A semi-Sasakian Lie group is a semi-contact Lie group $(G,\theta,\hat\Omega)$ with a left invariant complex structure $I$ on $\R^{>0}\ltimes_\theta G$ such that $(\R^{>0}\ltimes_\theta G,\hat\Omega,I)$ is a K\"ahler manifold.
			
			Equivalently, a semi-Sasakian Lie is a Lie group $G$ equipped with an automorphism $\theta$ and a K\"ahler structure $(g,I)$ on $\R^{>0}\ltimes_\theta G$ such that $I$ is $\R^{>0}\ltimes_\theta G$-invariant, $g$ is $G$-invariant and satisfies $\lambda_q^* g= q^2g$, where $\lambda_q$ is as above.
		\end{defin}

		\begin{cor}
			Let $\g$ be a Lie algebra of the invariant vector fields on a Lie group $G$, $D$  the derivation of $\g$, $\theta=\exp D$  a corresponding automorphism of $G$, $I$ a complex structure on $\R\ltimes \g$, $\omega$ and $\eta$ left invariant $2$-form and $1$-form correspondingly. Then the following conditions are equivalent:
			\begin{itemize}
				\item [(i)] $(\g,D,\omega,\eta,I)$ is a semi-Sasakian Lie algebra.
				\item[(ii)] $\left(\R^{>0}\ltimes_\theta G, \Omega=\omega+t^{-1}\alpha\wedge dt,2t^{-1}dt,I \right)$ is an lck Lie group.
				\item[(iii)] $(G,\theta, \hat\Omega=t^2\omega+tdt\wedge\eta,I)$ is a semi-Sasakian Lie group.

			\end{itemize}
		\end{cor}
		\begin{rem}
			Let $(G,\theta,g,I)$ be a semi-Sasakian Lie group. Let $\rho$ be a vector field corresponding to the left action of $\R^{>0}$ on $\R\ltimes_\theta G$. Define $s:\R^{>0}\ltimes_\theta G\to \R^{>0}\ltimes_\theta G$ by the rule 
			$$
			s(x)=g(\rho,\rho).
			$$
			Then $M=\{x\in \R\ltimes_\theta G|\ g(\rho,\rho|_M)=1\}$ is a Sasakian manifold. Moreover, we have an isomorphism of manifolds
			$$
			\alpha :\R^{>0}\ltimes_\theta G\to \R^{>0}\times G \ \ \ \alpha(p)=s(p)\times \left(\gamma_p\cap M\right),
			$$
			where $\gamma_p$ be an integral curve of $\rho$ containing $p$.
		Then, by Proposition 3.5 we have
			$$
			(\R^{>0}\ltimes_\theta G,g)\simeq (\R^{>0}\times M, g=s^2g_M+ds^2)
			$$
			where $g_M= g|_M$. Note that $M\subset G\times_\theta \R^{>0}$ is not necessary a subgroup. 
		\end{rem}

	\section{Projective Hessian Lie groups}
	\begin{defin}
		{\bfseries A Hessian Lie group} $(G,\nabla, g)$ is a Lie group $G$ endowed with a left invariant affine structure $\nabla$ and a left invariant Hessian metric $g$.
	\end{defin}
	\begin{theorem}
	Let $(G,\nabla,g)$ be an $n$-dimensional simply connected Hessian Lie group and $\theta$ the linear part of the corresponding affine action of $G$. Then there exists a left invariant Kähler metric on $G_\nabla=G\ltimes_\theta \R^n=TG.$
	\end{theorem}	
	\begin{proof}
		The group $G\ltimes_\theta \R^n=TG$ is locally biholomorphic to $\mathfrak{g} \oplus I \mathfrak{g}$. Hence, we have the local coordinates $x_1,\ldots,x_n,y_1,\ldots,y_n$ on $G\ltimes_\theta \R^n=TG$ such that $I\frac{\partial}{\partial x_i}=\frac{\partial}{\partial y_i}$ and $x_1,\ldots,x_n$ are constant along any fiber of $\pi:TG\to G$. The Hessian metric $g$ is locally equivalent to $\text{Hess}\varphi$. Define the associated Hermitian metric $g^T$ as in Section 2. By Proposition 2.3, $g^T$ is a Hessian metric which is locally expressed by $\text{Hess}_\mathbb{C}(4\pi^*\varphi)$. The subgroup $\text{id}\ltimes \R^n \subset G\ltimes_\theta \R^n$ acts on fibers of $TG\to G$. The function $\pi^*\varphi$ is constant along the fibers hence $g^T=\text{Hess}_\mathbb{C}(4\pi^*\varphi)$ is invariant under the action of  $\text{Id}\ltimes_\theta \R^n \subset G\ltimes_\theta \R^n$. By Proposition 2.5, 
		$$
		g^T(X,Y)=\pi^*g(X,Y)+\pi^*g(IX,IY)+\sqrt{-1}\pi^*g(IX,Y)-\sqrt{-1}\pi^*g(X,IY).
		$$
		Moreover, $g$ is invariant under the action of $G\ltimes 0\subset G\ltimes_\theta \R^n$. Thus, $g^T$ is invariant under the action of $G\ltimes 0\subset G\ltimes_\theta \R^n$, also. Therefore, $g^T$ is invariant under the action of the group $G\ltimes_\theta \R^n$.
	\end{proof}
	
	\begin{defin}
		A Lie algebra $\mathfrak{g}$ is called {\bfseries projective} if there is an invariant affine structure $\nabla$ on $\mathfrak{g}\times \R$  such that 
		$$
		\nabla_X E =\nabla_E X = X,
		$$
		where $X\in \mathfrak{g}$ and $E\in \R$. A Lie group is called {\bfseries projective} if the corresponding Lie algebra is projective.  
	\end{defin}
	
	Note that if $G$ is a projective Lie group then there is an invariant affine structure on $G\times \R^{>0}$.
	
	\begin{defin}
		A {\bfseries projective Hessian Lie group} $(G,g_G)$ is a projective Lie group $G$ endowed with a left-invariant Riemannian metric $g_G$ such that $(G\times\R^{>0},\nabla,g=t^2g_G +dt^2)$ is a Hessian manifold, where $t$ is a coordinate on $\R^{>0}$ and $\nabla$ be the affine connection on $G\times \R^{>0}$ corresponding to the projective structure on $G$. A Lie algebra corresponding to projective Hessian Lie group is called a {\bfseries projective Hessian Lie algebra}.
		
	\end{defin}

	\begin{theorem}
		Let $(G, g_G$ be an $n$-dimensional simply connected projective Hessian Lie group and $\theta$ the linear part of the corresponding affine representation of $G \times \R^{>0}$. Then there exists a structure of a semi-Sasakian Lie group on  $G\ltimes_\theta \R^{n+1}$. Moreover, $G\ltimes_\theta \R^{n+1}\simeq 
		TG\times \R$. 
	\end{theorem}
	
	\begin{lemma}
		Let $F,G,H$ be groups, $\alpha$ and $\beta$ actions of $F$ and $G$ on $H$ respectively such that for any $f\in F, g\in G$, and $h\in H$ we have
		$$
		\alpha(f)\beta(g)h=\beta(g)\alpha(f)h.
		$$
		Then 
		$$
		(F\times G)\ltimes_{\alpha \times \beta} H=F\ltimes_{\text{id} \times \alpha} (G\ltimes_\beta H). 
		$$
	\end{lemma}
	\begin{proof}
		Both semidirect products are equal to $F\times G \times H$ as sets. Thus, it is enough to check that two multiplications coincide. Write multiplication on the group $(F\times G)\ltimes_{\alpha \times \beta} H$
		$$
		((f_1\times g_1)\ltimes_{\alpha \times \beta} h_1)((f_2\times g_2)\ltimes_{\alpha \times \beta} h_2)=(f_1 f_2\times g_1 g_2)\ltimes_{\alpha \times \beta} h_1\alpha(f_1)\beta(g_1)h_2.
		$$
		Write multiplication on the group $F\ltimes_{\text{id} \times \alpha} (G\ltimes_\beta H)$
		$$
		(f_1\ltimes_{\text{id}\times \alpha} (g_1\ltimes_{\beta} h_1))(f_2\ltimes_{\text{id}\times \alpha} (g_2 \ltimes_ {\beta} h_2))=f_1 f_2\ltimes_{\text{id}\times \alpha} (g_1 \ltimes_{\beta} h_1)(g_2 \ltimes_{\beta} \alpha(f_1)h_2)=
		$$
		$$
		=f_1 f_2\ltimes_{\text{id}\times \alpha} (g_1 g_2 \ltimes_{\beta} h_1\beta(g_1)\alpha(f_1)h_2)=f_1 f_2\ltimes_{\text{id}\times \alpha} (g_1 g_2 \ltimes_{\eta} h_1\alpha(f_1)\beta(g_1)h_2).
		$$
		Two multiplications coincide.
		
	\end{proof}
	\begin{proof}[Proof of Theorem 6.5]
		There exist an invariant torsion free flat connection $\nabla$ on $G\times \R^{>0}$ and a Hessian metric $g$ invariant under the left action of $G$. Define a Kähler metric $g^T$ on $(G\times \R^{>0}) \ltimes_\theta \R^{n+1}$ as in the proof of Theorem 6.2. This metric is invariant under the left action of $(G\times \text{Id})\ltimes_\theta \R^{n+1}\subset (G\times \R^{>0}) \ltimes_\theta \R^{n+1}$ by the same argument as in the proof of Theorem 6.2. Moreover,
		$$
		g^T(X,Y)=\pi^*g(X,Y)+\pi^*g(IX,IY)+\sqrt{-1}\pi^*g(IX,Y)-\sqrt{-1}\pi^*g(X,IY),
		$$	
		where $\pi : G\times\R^{>0}\ltimes_\theta \R^{n+1}\simeq T(G\times \R^{>0}) \to G\times \R^{>0}$ is a projection. Let
		$$
		\lambda_q :G\times \R^{>0} \to G\times \R^{>0}, \ \ \lambda_q(p\times r)= p\times qr
		$$	
		and
		$$
		\mu_q :\left(G\ltimes_\theta \R^{>0}\right)\ltimes_\theta \R^{n+1} \to \left(G\ltimes_\theta \R^{>0}\right)\ltimes_\theta \R^{n+1}, \ \ \lambda_q((p\times r)\ltimes_\theta v)= (p\times qr)\ltimes_\theta \theta(q)v.
		$$
		Then we have the commutative diagram 
		$$
		\begin{CD}
		\left(G\ltimes_\theta \R^{>0}\right)\ltimes_\theta \R^{n+1} @>\mu_q>> \left(G\ltimes_\theta \R^{>0}\right)\ltimes_\theta \R^{n+1}\\
		@VV\pi V @VV\pi V  @.\\
		G\times\R^{>0} @>\lambda_q>> G\times\R^{>0}
		\end{CD},
		$$
		By the same argument as in Proposition 3.6, we get 
		\begin{equation}
		\mu_q^*g^T=q^2g^T.
		\end{equation}

		We constructed the Kähler metric $g^T$ on $(G\times \R^{>0}) \ltimes_\theta \R^{n+1}$ which is invariant under the action of $G\ltimes_\theta \R^{n+1}$ and satisfies (6.1). Also, by Lemma 6.6, we have 
		$$
		(G\times \R^{>0}) \ltimes_\theta \R^{n+1}= \R^{>0}\ltimes (G\ltimes_\theta \R^{n+1}).
		$$ Thus, $G\ltimes_\theta \R^{n+1}$ is a semi-Sasakian Lie group.
	\end{proof}

	Examples of projective Hessian Lie groups are described in the next two sections.

	\section{Regular convex cones} 
				
				\begin{defin}
					A subset $V\subset \R^n$ is called {\bfseries regular} if $V$ does not contain any straight line.
				\end{defin}

	Let $V\subset \R^n$ be a convex regular domain. We denote the maximal subgroup of $\text{GL}(\R^n)$ preserving $V$ by $\text{Aut}(V)$. Note that if $V$ is a regular convex cone then $$\text{Aut}(V)=(\text{Aut}(V)\cap \text{SL}(\R^n))\times \R^{>0}.$$ 

	The following theorem summarized known results.

	\begin{theorem}[\cite{vinb}, \cite{VGP}, \cite{shima}]
		Let $V \subset \R^n$ be a convex homogeneous regular domain, $${U=\R^n \oplus \sqrt {-1} V \subset \mathbb{C}^n},$$ and 
		$$
		\pi : U=\R^n \oplus \sqrt {-1} V \to \sqrt {-1} V \simeq V 
		$$ 
		be a projection. Then there exist a function $\varphi$ on $V$  and a subgroup $T\subset\text{Aut}(V)$ such that the following conditions are satisfied:	
				\begin{itemize}
				\item[(i)] $g_{can}=\text{Hess}\varphi$ is $\text{Aut}(V)$-invariant Hessian metric on $V$.
								
				\item[(ii)] $T$ acts on $V$ simply transitively.
				
				\item[(iii)] A group $\text{Aut}(V)\ltimes \R^n$ acts on $U$ by holomorphic automorphisms. Moreover, the subgroup $T\ltimes \R^n \subset \text{Aut}(V)\ltimes \R^n$ acts on $U$ simply transitively.
				
				\item[(iv)] The bilinear form $\text{Hess}_\CC (4\pi^* \varphi)$ is a $\text{Aut}(V)\ltimes \R^n$-invariant K\"ahler metric on $U$.
				
				\end{itemize}
	\end{theorem}
	See \cite{vinb} for (i) and (ii), \cite{VGP} for (iii), \cite{shima} for (iv).
	\begin{theorem}
			
			Let $V\subset\R^n$ be a convex homogeneous regular cone and $U$, $\pi$, $T$, $\varphi$ as in Theorem 7.2. Denote ${T_\text{SL}=T\cap \text{SL}(\R^n)}$. Then following conditions are satisfied:
			
			\begin{itemize}
				\item[(i)] The exists a $\text{Aut}_\text{SL}(V)$-invariant conical Hessian metric $g_{con}=\text{Hess}\varphi$ on $V$. The dilation subgroup $\R^{>0}\subset \text{Aut}(V)$ acts on $g_{con}$ by the rule $\lambda_q^* g_{con} = q^2 g_{con}$, for any $q\in \R^{>0}$. 
				
				\item[(ii)] The bilinear form $\text{Hess}_\CC (4\pi^* \varphi)$ is a $T_\text{SL}$-invariant K\"ahler metric. Moreover, $\lambda_q^* g = q^2 g$.
			\end{itemize}
		
	\end{theorem}		
	\begin{proof}
	This theorem is similar to the previous one. The metric $g_{con}$ on a convex regular cone $V$ is defined as $\text{Hess}(\varphi)$,	where $\varphi$ is a characteristic function of the cone (see \cite{vinb}). The item (ii) is analogous to the item (iv) from Theorem 7.2.  
	\end{proof}

	 By the item (ii) from Theorem 7.2, $T$ is a Hessian Lie group. By the item (iv) from Theorem 7.2, $T\ltimes \R^n$ is a K\"ahler Lie group. On the other hand, we can get the same result using Theorem 6.2. Thus, Theorem 6.2 generalizes the item (iv) from Theorem 7.2. 
	 
	 Analogically, by Theorem 7.3, $T_\text{SL}$ is a projective Hessian Lie group and $T_\text{SL}\ltimes \R^n$ is a Sasakian Lie group. We can get the same result using Theorem 6.5. Thus, Theorem 6.5 generalizes known constructions for homogeneous regular cones.

		\begin{defin}
			The group $T$ from the Theorem 7.2 is called a {\bfseries Lie group associated with $V$}. The corresponding Lie algebra $\mathfrak{t}$ is a {\bfseries Lie algebra associated with $V$}.
		\end{defin}

		\begin{proposition}[\cite{vinb}]
			Let $T_\text{SL}$ be a Lie group from Theorem 6.3 then there exist a convex regular domain $W$ such that $T_\text{SL}$ is an associated with $W$ Lie group. Conversely, if $S$ is a Lie group associated with a regular convex domain $W$ then there exist regular convex $V$ such that $S=T_\text{SL}$, where $T_\text{SL}$ is as in Theorem 7.3. 
		\end{proposition}
		
		\begin{cor}
			Let $T$ be a Lie group associated with a homogeneous regular domain. Then $T$ is both a Hessian Lie group and a projective Hessian Lie group.
		\end{cor}

						\begin{example}
							Let $V$ be the vector space of all real symmetric matrices of rank $n$ and $\Omega$ the set of all positive definite symmetric matrices in $V$. Then $\Omega$ is a regular convex cone and the group of upper triangular matrix $\text{T}(\R^n)$ acts simply transitively on $\Omega$ by $s(x)= s x s^T$, where $x \in \Omega$ and $s \in \text{T}(n)$.The characteristic function is equal to
							$$
							\varphi(x)=(\det x)^{-\frac{n+1}{2}}\varphi (e),
							$$ 
							where $e$ is the unit matrix (see \cite{shima}).
						\end{example}

	\section{Projective Hessian structure on $\text{SO}(2)$ and $\text{SU}(2)$}
	
		\begin{example}
			Consider the group $\R$ as the universal covering of $\text{U}(1)=\text{SO(2)}$. The identification $\text{SO}(2)\times{\R^{>0}} \simeq \R^2 \backslash \{0\}$ sets a projective Hessian structure on $\text{SO}(2)$ and on the universal covering $\R$. The corresponding to $\R$ semi-Sasakian group $\R\ltimes \R^2$ is the universal covering to the group of Euclidean motions $\text{E}(2)=\text{SO}(2)\ltimes \R^2$. Hence, the Lie algebra of Euclidean motions $\mathfrak{e}(2)$ is semi-Sasakian. Let $t, x, y, r$ be standard coordinates on the group $\R\ltimes\R^2\times \R^{>0}$. The Kähler metric is given by $\text{Hess}_\mathbb{C} (r^2)$. The Kähler structure on $\R\ltimes\R^2\times \R^{>0}$ induces the Kähler structure on $\text{SO}(2)\ltimes\R^2\times \R^{>0}$. Thus, the group of Euclidean motions $\text{E}(2)$ is semi-Sasakian, too.  
			
		\end{example}
		
	All $3$-dimensional Sasakian Lie algebras are classified.
	
	\begin{proposition}[\cite{5sasaki}]
		Any 3-dimensional Sasakian Lie algebra is isomorphic to one of the
		following: $\mathfrak{su}(2), \mathfrak{sl}(2,\R),\mathfrak{aff}(\R)\times \R$, and the Heisenberg algebra $\mathfrak{h}_3$.
	\end{proposition}

	The algebras  $\mathfrak{su}(2)$ and $\mathfrak{sl}(2,\R)$ are semisimple, and the algebras $\mathfrak{aff}(\R)\times \R$ and $\mathfrak{h}_3$ are nilpotent. The algebra $\mathfrak{e}(2)$ is solvable but not nilpotent. Therefore, the algebra  $\mathfrak{e}(2)$  is semi-Sasaian but not Sasakian.

	\begin{example}
		There is an identifications $\text{SU}(2)\simeq S^3$ and $\text{SU}(2)\times \R^{>0}=\R^4 \backslash \{0\}$. The group structure on $S^3$ equals to the restriction on $S^3$ of the standard
		$\text{SU}(2)$-action on $\mathbb{C}^2\simeq\R^4$. The corresponding semi-Sasakian Lie group is equal to $\text{SU}(2)\ltimes \mathbb{C}^2$.  
	\end{example}

	Notice that the group $\text{SU}(2)$ is not Hessian just because the sphere $S^3$ does not admit an affine structure. However, there exists an invariant affine structures on $\text{SU}(2)\times \R^{>0} \simeq \R^4 \backslash \{0\}$. Moreover, if the group $G$ is one of the previous examples of projective Hessian Lie groups (a clan or $\text{U}(1)$) then the group $G\times \R^{>0}$ is Hessian. We prove that the group $\text{SU}(2)\times \R^{>0}$ is not Hessian.

	\begin{proposition}
		The manifold $S^3\times \R^{>0}$ does not admit an $\R^{>0}$-invariant Hessian structure. 
	\end{proposition}
	\begin{proof}
		If $S^3\times \R^{>0}$ admits an $\R^{>0}$-invariant Hessian structure then the Hopf manifold 
		$$S^3\times \R^{>0} /_{(x\times t) \sim (x\times 2t)} \simeq S^3\times S^1
		$$ 
		admits a Hessian structure. In \cite{shima}, Shima proved that if $M$ is compact Hessian manifold then the universal covering $\tilde{M}$ is convex domain in $\R^n$. The universal covering of the Hopf manifold is diffeomorphic to $S^3\times \R^{>0}$ that cannot be diffeomorphic to any domain in $\R^4$. Therefore, $S^3\times \R^{>0}$ does not admit an $\R^{>0}$-invariant Hessian structure.
	\end{proof}
	
	\begin{cor}
		The group $\text{SU}(2)\times \R^{>0}$ is not Hessian.
	\end{cor}
	
	\paragraph*{Acknowledgements.} I would like to thank M. Verbitsky for fruitful discussions, and {D.V. Alekseevsky}, for his useful comments and help with preparation of the paper.

\end{document}